\documentclass[12pt]{article}
\usepackage{e-jc}

\usepackage{amsthm,amsmath,amssymb}

\usepackage{graphicx}

\usepackage[colorlinks=true,citecolor=black,linkcolor=black,urlcolor=blue]{hyperref}

\newcommand{\arxiv}[1]{\href{http://arxiv.org/abs/#1}{\texttt{arXiv:#1}}}

\theoremstyle{plain}
\newtheorem{theorem}{Theorem}
\newtheorem{lemma}[theorem]{Lemma}
\newtheorem{corollary}[theorem]{Corollary}

\theoremstyle{definition}

\theoremstyle{remark}
\newtheorem{remark}[theorem]{Remark}

\title{\bf Convolution estimates and the number of disjoint partitions}

\author{Paata Ivanisvili\thanks{This paper is  based upon work supported by the National Science Foundation under Grant No. DMS-1440140 while the author was in residence at the Mathematical Sciences Research Institute in Berkeley, California, during the Spring 2017 semester.}\\
\small Department of Mathematics, Kent State University\\[-0.8ex]
\small 1300 Lefton Esplanade, Kent OH 44242\\[-0.8ex] 
\small\tt pivanisv@kent.edu
}

\date{\dateline{May 27, 2017}{May 27, 2017}}

\begin{document}

\maketitle


\begin{abstract}
Let $X$ be a finite collection of sets. We count the number of ways  a disjoint union of $n-1$ subsets in $X$ is a set in $X$, and estimate this number from above by    $|X|^{c(n)}$ where 
$$
c(n)=\left(1-\frac{(n-1)\ln (n-1)}{n\ln n} \right)^{-1}.
$$ 
This extends the recent result of Kane--Tao, corresponding to the case $n=3$ where $c(3)\approx 1.725$, to an arbitrary finite number of disjoint  $n-1$ partitions. 

\end{abstract}

\section{Introduction}

Let $\{0,1\}^{m}$ be the Hamming cube of  dimension $m \geq 1$.  Set $1^{m}:=(1,1,\ldots, 1)$ to be the corner of $\{0,1\}^{m}$. Take a finite number of functions $f_{1}, \ldots, f_{n} : \{0,1\}^{m} \to \mathbb{R}$, and define the convolution at the corner $1^{m}$ as  
\begin{align*}
f_{1}*f_{2}*\ldots*f_{n}(1^{m}):=\sum_{x_{j} \in \{0,1\}^{m}\, :\, x_{1}+\cdots+x_{n}=1^{m}} f_{1}(x_{1})\cdots f_{n}(x_{n}). 
\end{align*} 
Given $f :\{0,1\}^{m} \to \mathbb{R}$ define its $L^{p}$ norm ($p\geq 1$)  in a standard way 
\begin{align*}
\|f\|_{p} := \left( \sum_{x \in \{0,1\}^{m}} |f(x)|^{p}\right)^{1/p}.
\end{align*}
For  $n\in \mathbb{N}$ we set 
\begin{align*}
p_{n} := \frac{\ln \frac{n^{n}}{(n-1)^{n-1}}}{\ln n}.
\end{align*}
Our main result is the following theorem 
\begin{theorem}\label{taot}
For any $n,m \geq 1$, and any $f_{1}, \ldots, f_{n} :\{0,1\}^{m} \to \mathbb{R}$ we have 
\begin{align}\label{maininequality}
f_{1}*f_{2}*\ldots*f_{n}(1^{m})\leq \prod_{j=1}^{n} \|f_{j}\|_{p_{n}}.
\end{align}
Moreover,  for each fixed $n$ exponent $p_{n}$ is the best possible in the sense that it cannot be replaced by any larger number. 
\end{theorem}
As an immediate application we obtain the following corollary (see Section~\ref{kak} below).
\begin{corollary}\label{taoc}
Let $X$ be a finite collection of sets. Then 
\begin{align}\label{bolof}
\left| \left\{ (A_{1}, \ldots, A_{n-1}, A) \in \underbrace{X\times \cdots \times X}_{n}\; :\ A = \bigsqcup_{j=1}^{n-1} A_{j}  \right\} \right|\leq |X|^{\frac{n}{p_{n}}},
\end{align}
where $\bigsqcup$ denotes  the disjoint union, and $|X|$ denotes cardinality of the set. 
\end{corollary}
The  corollary extends a recent result of Kane--Tao~\cite{KT}, corresponding to the case $n=3$ where 
$\frac{3}{p_{3}} \approx 1.725$, to an arbitrary finite number $n \geq 3$ disjoint partitions.  
\section{The proof of the theorem}
Following~\cite{KT} the proof goes by induction on the dimension of the cube $\{0,1\}^{m}$. 
The case $m=1$, which  is the most difficult, is the main  contribution of the current paper. 
\subsection{Basis: $m=1$}
 In this case, set $f_{j}(0)=u_{j}$ and $f_{j}(1)=v_{j}$ for $j=1,\ldots, n$.  Then the inequality (\ref{maininequality}) takes the form 
\begin{align}\label{i1}
\sum_{j=1}^{n} u_{j} \prod_{\substack{i=1 \\ i\neq j }}^{n}v_{i} \leq \prod_{j=1}^{n} \left(|u_{j}|^{p_{n}}+|v_{j}|^{p_{n}} \right)^{1/p_{n}}.
\end{align}

We do encourage the reader first to try to prove (\ref{i1}) in the case $n=3$, or visit~\cite{KT}, to see what is the obstacle. For example, when $n=3$  equality in (\ref{i1}) is attained at several points. Besides, direct differentiation of (\ref{i1}) reveals many ``bad'' critical points at which finding the values of (\ref{i1}) would require numerical computations~\cite{KT}. The number of critical points together with  equality cases increases as $n$ becomes larger, therefore one is forced to come up with a different idea. We will overcome this obstacle by looking at (\ref{i1}) in dual coordinates. 

Without loss of generality we can assume that $u_{j}$ and $v_{j}$ are nonnegative for $j=1,\ldots, n$. Moreover, we  can assume that $v_{j}\neq 0$ for all $j$ otherwise the inequality (\ref{i1}) is trivial.

Let us divide (\ref{i1}) by $\prod_{j=1}^{n} v_{j}$. Denoting $x_{j}:=(u_{j}/v_{j})^{p_{n}}$ we see that it is enough to prove the following lemma. 

\begin{lemma}
For any $n\geq 2$ and all  $x_{1}, \ldots, x_{n} \geq  0$ we have 
\begin{align}\label{mine}
\left(\sum_{i=1}^{n} x^{1/p_{n}}_{i}\right)^{p_{n}}  \leq \prod_{i=1}^{n} (1+x_{i}),
\end{align}
where $p_{n} = \frac{\ln \frac{n^{n}}{(n-1)^{n-1}}}{\ln n}$
\end{lemma}
\begin{proof}
For $n=2$ the lemma is trivial.    By induction on $n$,   monotonicity of the map 
\begin{align*}
p \to \left(\sum_{i=1}^{n} x^{1/p}_{i}\right)^{p},
\end{align*}
and the fact that $p_{n}$ is decreasing, we can assume that all $x_{i}$ are strictly positive. 
For convenience we set $p:=p_{n}$. 
Introducing new variables we rewrite (\ref{mine}) as follows 
\begin{align*}
p \ln \left( \sum x_{i} \right) \leq \sum \ln (1+x_{i}^{p}).
\end{align*}
 Concavity of the function $\ln (x)$ provides us with a simple representation of the logarithmic function  
\begin{align*}
\ln(x) = \min_{b \in \mathbb{R}} (b+e^{-b}x-1).
\end{align*}
Therefore we are left to show that for all $x_{i} >0$ and all $b_{i} \in \mathbb{R}$ we have 
\begin{align*}
B(x,b):=  \sum (b_{i}+(1+x_{i}^{p})e^{-b_{i}}-1)-p \ln \left( \sum x_{i} \right) \geq 0,
\end{align*}
where $x=(x_{1}, \ldots, x_{n})$ and $b = (b_{1}, \ldots, b_{n})$. Notice that given a vector $b \in \mathbb{R}^{n}$, the  infimum of $B(x,b)$ in $x$ cannot be reached at infinity because of the slow growth of the logarithmic function. Therefore, we look at critical points of  $B$ in $x$ 
\begin{align*}
x^{*}_{k} = \frac{e^{\frac{b_{k}}{p-1}}}{\left(\sum_{i} e^{\frac{b_{i}}{p-1}} \right)^{\frac{1}{p}}} \quad \text{for} \quad k=1, \ldots, n. 
\end{align*}
Notice that $\sum x^{*}_{i} = \left( \sum_{i} e^{\frac{b_{i}}{p-1}} \right)^{\frac{p-1}{p}}$. Therefore 
\begin{align*}
B(x^{*}, b) = \sum_{k} (b_{k} + e^{-b_{k}})+1-n-(p-1)\ln \left(\sum_{k} e^{\frac{b_{k}}{p-1}} \right).
\end{align*}
Setting $r:=p-1>0$, and introducing new variables again we are left to show that 
\begin{align*}
f(y):=1-n+\sum \ln y^{r}_{i} +  \sum \frac{1}{y_{i}^{r}}  -  r \ln \left(\sum y_{i} \right) \geq 0
\end{align*}
for all $y_{i}>0$. It is straightforward to check that $f(y) \geq 0$ on the diagonal, i.e., when $y_{1}=y_{2}=\ldots=y_{n}$.

 In general, we notice that critical points of $f(y)$ satisfy the equation 
\begin{align}\label{crit}
\frac{1}{y_{i}} - \frac{1}{y_{i}^{r+1}} = \frac{1}{y_{j}}- \frac{1}{y_{j}^{r+1}}=\frac{1}{\sum y_{k}}.
\end{align}
Equation (\ref{crit}) gives the identity $\sum y_{i}^{-r}=n-1$, and so at critical points (\ref{crit}) we are only left to show 
\begin{align}\label{lll}
\sum \ln y_{i}   -  \ln \left( \sum y_{i} \right)\geq 0. 
\end{align}
Since the mapping 
\begin{align*}
s \to  \frac{1}{s} - \frac{1}{s^{r+1}}, \quad s>0,
\end{align*}
is increasing on $(0, (1+r)^{1/r})$ and decreasing on the remaining part of the ray, we can assume without loss of generality that $k$ numbers of $x_{i}$ equal to $u\geq (1+r)^{1/r}$, and the remaining $n-k$ numbers of $x_{i}$ equal to $v\leq (1+r)^{1/r}$. Moreover, we can assume that $0<k<n$ otherwise the statement is already proved. From (\ref{crit}), we have 
\begin{align}\label{eq1}
\frac{1}{u}-\frac{1}{u^{r+1}} = \frac{1}{v} - \frac{1}{v^{r+1}} = \frac{1}{k u + (n-k)v}.
\end{align}
From the equality of the first and the third expressions in (\ref{eq1}) it follows that 
\begin{align}\label{expv}
v = \frac{u^{r+1}(1-k)+ku}{(u^{r}-1)(n-k)}.
\end{align}
In order $v$ to be positive we assume that the numerator of (\ref{expv}) is non negative. If we plug the expression for $v$ from (\ref{expv}) into the first equality of (\ref{eq1}) then after some simplifications we obtain the following equation in the variable    $z:=u^{r} \geq 1+r$ 
\begin{align}\label{yrb}
\frac{(z-1)^{r} (n-k)^{r+1}}{(z(1-k)+k)^{r}} = (n-1)z-k.
\end{align}
It follows from (\ref{eq1}) that $(ku+(n-k)v)^{r} = \left(\frac{z}{z-1}\right)^{r} z$, and so using (\ref{yrb}) we obtain 
\begin{align*}
v^{r} = \frac{z (n-k)}{(n-1)z-k}.
\end{align*}
 Therefore at critical points (\ref{lll}) simplifies as follows 
\begin{align}\label{last}
(n-1) \ln z + (n-k) \ln \frac{n-k}{(n-1)z-k} - r \ln \frac{z}{z-1} \geq 0.
\end{align}
Now it is pretty straightforward to show that (\ref{last}) is non negative even under the assumption $z \geq 1+r$ for $r = p-1 = \frac{(n-1) \ln \frac{n}{n-1}}{\ln n}$. Indeed, notice that $z > \frac{n}{n-1}$, and  the map 
\begin{align*}
k \to (n-k)\ln \frac{n-k}{(n-1)z-k} 
\end{align*}
is increasing on $[1,n-1]$. Therefore it is enough to check nonnegativity of (\ref{last}) when $k=1$, in which   case the inequality follows again using $z > \frac{n}{n-1}$, and the fact that the map 
\begin{align*}
z \to \frac{\ln \left(1+\frac{1}{z(n-1)-1} \right)}{\ln \left(1+\frac{1}{z-1} \right)}
\end{align*}
is increasing for $z \geq \frac{n}{n-1}$. 
\end{proof}
\begin{remark}
Choice $x_{1}=\ldots=x_{n}=\frac{1}{n-1}$ gives equality in \textup{(\ref{mine})}, and this confirms the fact that $p_{n}$ is the best possible in Theorem~\textup{\ref{taot}}. 
\end{remark}

\subsection{Inductive step}
Inductive step is the same as in~\cite{KT} without any modifications. This is a standard argument for obtaining estimates on the Hamming cube (see for example~\cite{IV}). In order to make the paper self contained we decided to repeat the argument. 

Suppose (\ref{maininequality}) is true on the Hamming cube of dimension $m$. Without loss of generality assume $f_{j} \geq 0$, and set $g_{j} :=f_{j}^{p}$ for all $j$. Define 
\begin{align*}
B_{n}(y_{1}, \ldots, y_{n}) := y_{1}^{1/p_{n}}\cdots y_{n}^{1/p_{n}}.
\end{align*}
For $x_{j} \in \{0,1\}^{m+1}$, let $x_{j} = (\bar{x}_{j}, x'_{j})$ where $\bar{x}_{j}$ is the vector consisting of the first $m$ coordinates of $x_{j}$, and number  $x'_{j}$ denotes the last $m+1$ coordinate of $x_{j}$. Set 
\begin{align*}
\tilde{g}_{j}(x'_{j}):=\sum_{\bar{x}_{j} \in \{0,1\}^{m}} g_{j}(\bar{x}_{j}, x'_{j}) \quad j=1,\ldots, n. 
\end{align*}
We have 
\begin{align*}
&\sum_{x_{j} \in \{0,1\}^{m+1}\; :\; x_{1}+\cdots+x_{n}=1^{m+1}.}B(g_{1}(x_{1}),\ldots, g_{n}(x_{n}))=\\
&\sum_{x'_{j}\in\{0,1\}\; :\; x'_{1}+\cdots+x'_{n}=1.}\; \sum_{\bar{x}_{j} \in \{0,1\}^{m}\; :\; \bar{x}_{1}+\cdots+\bar{x}_{n}=1^{m}.}B(g_{1}(x_{1}),\ldots, g_{n}(x_{n}))\stackrel{\mathrm{induction}}{\leq}\\
& \sum_{x'_{j}\in\{0,1\}\; :\; x'_{1}+\cdots+x'_{n}=1.} B(\tilde{g}_{1}(x'_{1}), \ldots, \tilde{g}_{n}(x'_{n}))\stackrel{\mathrm{basis}}{\leq} \\
&B\left( \sum_{x_{1} \in \{0,1\}^{m+1}} g_{1}(x_{1}), \ldots, \sum_{x_{n} \in \{0,1\}^{m+1}} g_{n}(x_{n})\right).
\end{align*}
\subsection{The proof of Corollary~\ref{taoc}}\label{kak}
Without loss of generality we may assume that all the sets $A$ in $X$ are subsets of $\{1,\ldots, m\}$ with some natural $m \geq 1$ (see~\cite{KT}). For $j=1,\ldots, n$ define functions 
$$
f_{j} : \{0,1\}^{m} \to \{0,1\}
$$
 as follows:  
\begin{align*}
f_{1}(a_{1}, \ldots, a_{m})=\ldots=f_{n-1}(a_{1}, \ldots, a_{m})=1
\end{align*}
if the set $\{1\leq i \leq m\; : \; a_{i}=1\}$ lies in $X$, and $f_{j}=0$ otherwise. Finally we define 
\begin{align*}
f_{n}(a_{1}, \ldots, a_{m})=1
\end{align*}
if the set $\{1\leq i \leq m\; : \; a_{i}=0\}$ lies in $X$, and $f_{n}=0$ otherwise. 
Notice that in this case inequality  (\ref{maininequality}) becomes  (\ref{bolof}). 
\subsection*{Acknowledgements}
I would like to thank an anonymous referee, and Benjamin Jaye for  helpful comments.


\end{document}